\documentclass[a4paper, 12pt, reqno]{amsart}
\usepackage{latexsym}
\usepackage{amssymb,amsfonts,amsmath}

\begin{document}
\title{A note on the second moment of automorphic $L$-functions}
\author{H. M. Bui}
\address{Mathematical Institute, University of Oxford, OXFORD, OX1 3LB}
\email{hung.bui@maths.ox.ac.uk}

\begin{abstract}
We obtain the formula for the twisted harmonic second moment of the $L$-functions associated with primitive Hecke eigenforms of weight $2$. A consequence of our mean value theorem is reminiscent of recent results of Conrey and Young on the reciprocity formula for the twisted second moment of Dirichlet $L$-functions. 
\end{abstract}

\maketitle

\begin{section}
{Introduction}
\end{section}

In this paper, we study the twisted second moment of the family of $L$-functions arising from $\mathcal{S}_{2}^{*}(q)$, the set of primitive Hecke eigenforms of weight $2$, lever $q$ ($q$ prime). For $f(z)\in \mathcal{S}_{2}^{*}(q)$, $f$ has a Fourier expansion
\begin{equation*}
f(z)=\sum_{n=1}^{\infty}n^{1/2}\lambda_{f}(n)e(nz),
\end{equation*}
where the normalization is such that $\lambda_{f}(1)=1$. The $L$-function associated to $f$ has an Euler product
\begin{equation*}
L(f,s)=\sum_{n=1}^{\infty}\frac{\lambda_{f}(n)}{n^s}=\bigg(1-\frac{\lambda_{f}(q)}{q^s}\bigg)^{-1}\prod_{\substack{h\ \textrm{prime}\\h\ne q}}\bigg(1-\frac{\lambda_{f}(h)}{h^s}+\frac{1}{h^{2s}}\bigg)^{-1}.
\end{equation*}
The series is absolutely convergent when $\Re s>1$, and admits analytic continuation to all of $\mathbb{C}$. The functional equation for $L(f,s)$ is
\begin{equation*}
\Lambda(f,s):=\bigg(\frac{\sqrt{q}}{2\pi}\bigg)^{s}\Gamma(s+{\scriptstyle{\frac{1}{2}}})L(f,s)=\varepsilon_{f}\Lambda(f,1-s),
\end{equation*}
where $\varepsilon_f=-q^{1/2}\lambda_{f}(q)=\pm1$. We define the harmonic average as
\begin{equation*}
\sum_{f}{\!}^{h}\ A_{f}:=\sum_{f\in \mathcal{S}_{2}^{*}(q)}\frac{A_{f}}{4\pi(f,f)},
\end{equation*}
where $(f,g)$ is the Petersson inner product on the space $\Gamma_{0}(q)\backslash\mathbb{H}$.

We are interested in the twisted second moment of this family of $L$-functions. We define
\begin{equation*}
S(p,q)=\sum_{f\in \mathcal{S}_{2}^{*}(q)}{\!\!\!\!}^{h}\ L(f,{\scriptstyle{\frac{1}{2}}})^2\lambda_{f}(p).
\end{equation*}
Our main theorem is\\

\newtheorem{theorem}{Theorem}\begin{theorem}
Suppose $q$ is prime and $0<p\leq Cq$, for some fixed $C<1$. Then we have
\begin{equation*}
S(p,q)=\frac{d(p)}{\sqrt{p}}\log\frac{q}{4\pi^2p}+O(p^{1/2}q^{-1+\varepsilon}).
\end{equation*} 
\end{theorem}

\newtheorem{remark}{Remark}\begin{remark}
\emph{The twisted harmonic fourth moment has been considered by Kowalski, Michel and VanderKam [\textbf{\ref{KMV}}], where they gave an asymptotic formula for the fourth power mean value provided that $p\ll q^{1/9-\varepsilon}$.}
\end{remark}

\begin{remark}
\emph{In a similar setting, Iwaniec and Sarnak [\textbf{\ref{IS}}] have given the exact formula for the twisted second moment of the automorphic $L$-functions arising from $\mathcal{H}_{k}(1)$, the set of newforms in $\mathcal{S}_{k}(1)$, where $\mathcal{S}_k(1)$ is the linear space of holomorphic cusp forms of weight $k$. Precisely, they showed that for $k>2$, $k\equiv0(\textrm{mod}\ 2)$, and for any $m\geq1$, we have
\begin{eqnarray*}
&&\frac{12}{k-1}\sum_{f\in\mathcal{H}_{k}(1)}w_fL(f,{\scriptstyle{\frac{1}{2}}})^2\lambda_f(m)=2(1+i^k)\frac{d(m)}{\sqrt{m}}\bigg(\sum_{0<l\leq k/2}\frac{1}{l}-\log2\pi\sqrt{m}\bigg)\nonumber\\
&&\qquad-\frac{2\pi i^k}{\sqrt{m}}\sum_{h\ne m}d(h)d(h-m)p_k\bigg(\frac{h}{m}\bigg)+\frac{2\pi i^k}{\sqrt{m}}\sum_{h}d(h)d(h+m)q_k\bigg(\frac{h}{m}\bigg),
\end{eqnarray*}
where $p_k(x)$ and $q_k(x)$ are Hankel transforms of Bessel functions
\begin{equation*}
p_k(x)=\int_{0}^{\infty}Y_0(y\sqrt{x})J_{k-1}(y)dy,\textrm{ and }q_k(x)=\frac{2}{\pi}\int_{0}^{\infty}K_0(y\sqrt{x})J_{k-1}(y)dy.
\end{equation*}
Here the weight $w_f=\zeta(2)L(\textrm{sym}^2(f),1)^{-1}$, where the symmetric square $L$-function $L(\textrm{sym}^2(f),s)$ corresponding to $f$ is defined by
\begin{equation*}
L(\textrm{sym}^2(f),s)=\zeta(2s)\sum_{n=1}^{\infty}\frac{\lambda_f(n^2)}{n^s}.
\end{equation*}}
\end{remark}

In the context of Dirichlet $L$-functions, consider
\begin{equation*}
M(p,q)=\frac{1}{\varphi^{*}(q)}\sum_{\chi(\textrm{mod}\ q)}{\!\!\!\!\!\!}^{\textstyle{*}}\ |L({\scriptstyle{\frac{1}{2}}},\chi) |^{2}\chi(p),
\end{equation*}
where $\sum^{*}$ denotes summation over all primitive characters $\chi$(mod $q$), and $\varphi^{*}(q)$ is the number of primitive characters. This is the twisted second moment of Dirichlet $L$-functions. In a recent paper, Conrey [\textbf{\ref{C}}] proved that there is a kind of reciprocity formula relating $M(p,q)$ and $M(-q,p)$ when $p$ and $q$ are distinct prime integers. Precisely, Conrey showed that
\begin{equation*}
M(p,q)=\frac{\sqrt{p}}{\sqrt{q}}M(-q,p)+\frac{1}{\sqrt{p}}\bigg(\log\frac{q}{p}+A\bigg)+\frac{B}{2\sqrt{q}}+O\bigg(\frac{p}{q}+\frac{\log q}{q}+\frac{\log pq}{\sqrt{pq}}\bigg),
\end{equation*}
where $A$ and $B$ are some explicit constants. This provides an asymptotic formula for $M(p,q)-\sqrt{p/q}M(-q,p)$ under the condition that $p\ll q^{2/3-\varepsilon}$. The error term above was improved by Young [\textbf{\ref{Y}}] so that the asymptotic formula holds for $p\ll q^{1-\varepsilon}$.

We now take $p$ to be prime and, similarly as before, $S(q,p)$ is defined as the harmonic second moment, twisted by $\lambda_g(q)$, of the family of $L$-functions arising from $g(z)\in\mathcal{S}_{2}^{*}(p)$. We note that as $q$ is prime, the Ramanujan bound $|\lambda_{f}(n)|\leq d(n)$ [\textbf{\ref{D}}] yields
\begin{equation*}
S(q,p)\ll \sum_{g\in \mathcal{S}_{2}^{*}(p)}{\!\!\!\!}^{h}\ L(g,{\scriptstyle{\frac{1}{2}}})^2\ll\log p.
\end{equation*}
Thus as a trivial consequence of Theorem 1, for $p<q$ we have
\begin{equation*}
S(p,q)-\sqrt{p/q}S(q,p)=\frac{2}{\sqrt{p}}\log\frac{q}{4\pi^2p}+O(p^{1/2+\varepsilon}q^{-1/2}).
\end{equation*}
This leads to an asymptotic formula for $S(p,q)-\sqrt{p/q}S(q,p)$, at least for $p$ as large as $q^{1/2-\varepsilon}$. The results in the Dirichlet $L$-functions case [\textbf{\ref{C}},\textbf{\ref{Y}}] suggest that the asymptotic formula should hold for $p\ll q^{\theta}$, for any $\theta<1$. However, our technique fails to extend the range to any power $\theta>1/2$. For that purpose, we need more refined estimates for the off-diagonal terms of $S(p,q)$ and $S(q,p)$. The intricate calculations seem to suggest that there is a large cancellation between these two expressions. The nature of this is not well-understood.

\begin{section}
{Preliminary lemmas}
\end{section}

We require some lemmas. We begin with Hecke's formula for primitive forms.\\

\newtheorem{lemm}{Lemma}\begin{lemm}
For $m,n\geq1$,
\begin{equation*}
\lambda_{f}(m)\lambda_{f}(n)=\sum_{\substack{d|(m,n)\\(d,q)=1}}\lambda_{f}\bigg(\frac{mn}{d^2}\bigg).
\end{equation*}
\end{lemm}

The next lemma is a particular case of Petersson's trace formula.

\begin{lemm}
For $m,n\geq1$, we have
\begin{equation*}
\sum_{f\in \mathcal{S}_{2}^{*}(q)}{\!\!\!\!}^{h}\ \lambda_{f}(m)\lambda_{f}(n)=\delta_{m,n}-J_{q}(m,n),
\end{equation*}
where $\delta_{m,n}$ is the Kronecker symbol and
\begin{equation*}
J_{q}(m,n)=2\pi\sum_{c=1}^{\infty}\frac{S(m,n;cq)}{cq}J_{1}\bigg(\frac{4\pi\sqrt{mn}}{cq}\bigg).
\end{equation*}
Here $J_{1}(x)$ is the Bessel function of order $1$, and $S(m,n;c)$ is the Kloosterman sum
\begin{equation*}
S(m,n;c)=\sum_{a(\emph{mod}\
c)}{\!\!\!\!\!}^{\displaystyle{*}}\ e\bigg(\frac{ma+n\overline{a}}{c}\bigg).
\end{equation*}
Moreover we have
\begin{equation*}
J_{q}(m,n)\ll(m,n,q)^{1/2}(mn)^{1/2+\varepsilon}q^{-3/2}.
\end{equation*}
\end{lemm}

The above estimate follows easily from the bound $J_{1}(x)\ll x$ and Weil's bound on Kloosterman sums.

We mention a result of Jutila [\textbf{\ref{J}}] (cf. Theorem 1.7), which is an extension of the Voronoi summation formula.
\begin{lemm}
Let $f:\mathbb{R}^{+}\rightarrow\mathbb{C}$ be a $C^\infty$ function which vanishes in the neighbourhood of $0$ and is rapidly decreasing at infinity. Then for $c\geq1$ and $(a,c)=1$,
\begin{eqnarray*}
c\sum_{m=1}^{\infty}d(m)e\bigg(\frac{am}{c}\bigg)f(m)&=&2\int_{0}^{\infty}(\log\frac{\sqrt{x}}{c}+\gamma)f(x)dx\nonumber\\
&&\ -2\pi\sum_{m=1}^{\infty}d(m)e\bigg(\frac{-\overline{a}m}{c}\bigg)\int_{0}^{\infty}Y_{0}\bigg(\frac{4\pi\sqrt{mx}}{c}\bigg)f(x)dx\nonumber\\
&&\ +4\sum_{m=1}^{\infty}d(m)e\bigg(\frac{\overline{a}m}{c}\bigg)\int_{0}^{\infty}K_{0}\bigg(\frac{4\pi\sqrt{mx}}{c}\bigg)f(x)dx.
\end{eqnarray*}
\end{lemm}

The next lemma concerns the approximate functional equation for $L$-functions.

\begin{lemm}
Let $G(s)$ be an even entire function satisfying $G(0)=1$ and $G$ has a double zero at each $s\in\mathbb{Z}$. Furthermore let assume that $G(s)\ll_{A,B}(1+|s|)^{-A}$ for any $A>0$ in any strip $-B\leq\Re{s}\leq B$. Then for $f\in \mathcal{S}_{2}^{*}(q)$,
\begin{equation*}
L(f,{\scriptstyle{\frac{1}{2}}})^2=2\sum_{n=1}^{\infty}\frac{d(n)\lambda_{f}(n)}{\sqrt{n}}W_{q}\bigg(\frac{4\pi^{2}n}{q}\bigg),
\end{equation*}
where
\begin{equation*}
W_{q}(x)=\frac{1}{2\pi i}\int_{(1)}G(s)\Gamma(s+1)^{2}\zeta_{q}(2s+1)x^{-s}\frac{ds}{s}.
\end{equation*}
Here $\zeta_q(s)$ is defined by
\begin{equation*}
\zeta_q(s)=\sum_{\substack{n=1\\(n,q)=1}}^{\infty}n^{-s}\qquad(\sigma>1).
\end{equation*}
\end{lemm}
\begin{proof}
From Lemma 1 we first note that
\begin{equation*}
L(f,s)^2=\zeta_q(2s)\sum_{n=1}^{\infty}\frac{d(n)\lambda_f(n)}{n^s}\qquad(\sigma>1).
\end{equation*}
Consider
\begin{equation*}
A(f):=\frac{1}{2\pi i}\int_{(1)}\frac{G(s)\Lambda(f,s+{\scriptstyle{\frac{1}{2}}})^2}{\frac{\sqrt{q}}{2\pi}}\frac{ds}{s}.
\end{equation*}
Moving the line of integration to $\Re s=-1$, and applying Cauchy's theorem and the functional equation, we derive that $A(f)=L(f,{\scriptstyle{\frac{1}{2}}})^{2}-A(f)$. Expanding $\Lambda(f,s+{\scriptstyle{\frac{1}{2}}})^2$ in a Dirichlet series and integrating termwise we obtain the lemma.
\end{proof}

For our purpose, $W_q$ is basically a ``cut-off'' function. Indeed, we have the following.
\begin{lemm}
The function $W_q$ satisfies
\begin{equation*}
W_{q}^{(j)}(x)\ll_{j,N}x^{-N}\ \textrm{for }x\geq1\textrm{ and all }j,N\geq0,
\end{equation*}
\begin{equation*}
x^{i}W_{q}^{(j)}(x)\ll_{i,j}|\log x|\ \textrm{for }0<x<1\textrm{ and all }i\geq j\geq0,
\end{equation*}
and
\begin{equation*}
W_{q}(x)=-\bigg(1-\frac{1}{q}\bigg)\frac{\log x}{2}+\frac{\log q}{q}+O_N(x^{N}) \ \textrm{for }0<x<1\textrm{ and all }N\geq0.
\end{equation*}
The implicit constants are independent of $q$.
\end{lemm}
\begin{proof}
The first estimate is a direct consequence of Stirling's formula after differentiating under the integral sign and shifting the line of integration to $\Re s=N$. The only difference in the other two estimates is that one has to move the line of integration to $\Re s=-N$.
\end{proof}

\begin{section}
{Proof of Theorem 1}
\end{section}

Our argument in this section follows closely [\textbf{\ref{KM}}]. From Lemma 4 and Lemma 2 we obtain
\begin{equation*}
S(p,q)=2\frac{d(p)}{\sqrt{p}}W_{q}\bigg(\frac{4\pi^2p}{q}\bigg)-2R(p,q),
\end{equation*}
where
\begin{equation*}
R(p,q)=\sum_{n=1}^{\infty}\frac{d(n)}{\sqrt{n}}J_{q}(n,p)W_{q}\bigg(\frac{4\pi^2n}{q}\bigg).
\end{equation*}
Using Lemma 5, the first term is
\begin{equation*}
\frac{d(p)}{\sqrt{p}}\log\frac{q}{4\pi^2 p}+O(p^{-1/2}q^{-1+\varepsilon}+p^{1/2+\varepsilon}q^{-1}).
\end{equation*}
Thus, we are left to consider $R(p,q)$. We have
\begin{equation*}
R(p,q)=2\pi\sum_{n=1}^{\infty}\frac{d(n)}{\sqrt{n}}\sum_{c=1}^{\infty}\frac{S(n,p;cq)}{cq}J_{1}\bigg(\frac{4\pi\sqrt{np}}{cq}\bigg)W_{q}\bigg(\frac{4\pi^2n}{q}\bigg).
\end{equation*}
Using Weil's bound for Kloosterman sums and $J_1(x)\ll x$, the contribution from the terms $c\geq q$ is 
\begin{equation*}
\ll p^{1/2}q^{-3/2}\sum_{n=1}^{\infty}(n,p)^{1/2}d(n)\bigg|W_{q}\bigg(\frac{4\pi^2n}{q}\bigg)\bigg|\sum_{c\geq q}\frac{d(c)}{c^{3/2}}\ll p^{1/2}q^{-1+\varepsilon}.
\end{equation*}
Thus we need to study
\begin{equation*}
\frac{2\pi}{q}\sum_{c<q}\frac{1}{c}\sum_{a(\textrm{mod}\ cq)}{\!\!\!\!\!\!\!}^{*}\ \ e\bigg(\frac{\overline{a}p}{cq}\bigg)\sum_{n=1}^{\infty}d(n)e\bigg(\frac{an}{cq}\bigg)\frac{J_{1}(\frac{4\pi\sqrt{np}}{cq})W_{q}(\frac{4\pi^2n}{q})}{\sqrt{n}}.
\end{equation*}
We fix a $C^{\infty}$ function $\xi:\mathbb{R}^{+}\rightarrow[0,1]$, which satisfies $\xi(x)=0$ for $0\leq x\leq1/2$ and $\xi(x)=1$ for $x\geq1$, and attach the weight $\xi(n)$ to the innermost sum. Using Lemma 3, this is equal to
\begin{eqnarray*}
\frac{4\pi}{q^2}\sum_{c<q}\frac{1}{c^2}S(0,p;cq)\int_{0}^{\infty}(\log\frac{\sqrt{t}}{cq}+\gamma)J_{1}\bigg(\frac{4\pi\sqrt{tp}}{cq}\bigg)W_{q}\bigg(\frac{4\pi^2t}{q}\bigg)\xi(t)\frac{dt}{\sqrt{t}}-Y+K,
\end{eqnarray*}
where
\begin{eqnarray}
Y&=&\frac{4\pi^2}{q^2}\sum_{c<q}\frac{1}{c^2}\sum_{n=1}^{\infty}d(n)S(0,p-n;cq)\nonumber\\
&&\qquad\qquad\int_{0}^{\infty}Y_{0}\bigg(\frac{4\pi\sqrt{nt}}{cq}\bigg)J_{1}\bigg(\frac{4\pi\sqrt{tp}}{cq}\bigg)W_{q}\bigg(\frac{4\pi^2t}{q}\bigg)\xi(t)\frac{dt}{\sqrt{t}}\label{4},
\end{eqnarray}
and
\begin{eqnarray}
K&=&\frac{8\pi}{q^2}\sum_{c<q}\frac{1}{c^2}\sum_{n=1}^{\infty}d(n)S(0,p+n;cq)\nonumber\\
&&\qquad\qquad\int_{0}^{\infty}K_{0}\bigg(\frac{4\pi\sqrt{nt}}{cq}\bigg)J_{1}\bigg(\frac{4\pi\sqrt{tp}}{cq}\bigg)W_{q}\bigg(\frac{4\pi^2t}{q}\bigg)\xi(t)\frac{dt}{\sqrt{t}}.\label{5}
\end{eqnarray}
We will deal with $Y$ and $K$ in the next three lemmas. For the first sum, since $S(0,p;cq)=\mu(q)S(0,p\overline{q};c)$ and $J_{1}(x)\ll x$, this is
\begin{equation*}
\ll p^{1/2}q^{-3}\sum_{c<q}\frac{1}{c^2}\int_{1/2}^{\infty}\bigg|W_{q}\bigg(\frac{4\pi^2t}{q}\bigg)\bigg|(\log tcq)dt\ll p^{1/2}q^{-2+\varepsilon}.
\end{equation*}

\begin{lemm}
For $K$ defined as in \eqref{5}, we have
\begin{equation*}
K\ll p^{1/2}q^{\varepsilon}(q-p)^{-1+\varepsilon}.
\end{equation*}
And hence $K\ll p^{1/2}q^{-1+\varepsilon}$, given that $p\leq Cq$ for some fixed $C<1$.
\end{lemm}
\begin{remark}
\emph{This is the only place where the condition $p\leq Cq$ for some constant $C<1$ is used.}
\end{remark}
\begin{proof}
The integral involving $K_0$, using $K_{0}(y)\ll y^{-1/2}e^{-y}$, is
\begin{eqnarray*}
&&\int_{0}^{\infty}K_{0}\bigg(\frac{4\pi\sqrt{nt}}{cq}\bigg)J_{1}\bigg(\frac{4\pi\sqrt{tp}}{cq}\bigg)W_{q}\bigg(\frac{4\pi^2t}{q}\bigg)\xi(t)\frac{dt}{\sqrt{t}}\nonumber\\
&=&\frac{cq}{2\pi\sqrt{n}}\int_{0}^{\infty}K_{0}(y)J_{1}\bigg(\sqrt{\frac{p}{n}}y\bigg)W_{q}\bigg(\frac{c^2qy^2}{4n}\bigg)\xi\bigg(\frac{c^2q^2y^2}{16\pi^2n}\bigg)dy\nonumber\\
&\ll&\frac{cp^{1/2}q^{1+\varepsilon}}{n}\int_{\sqrt{n}/cq}^{\infty}y^{1/2}e^{-y}dy\ll\frac{cp^{1/2}q^{1+\varepsilon}}{n}e^{-\sqrt{n}/2cq}.
\end{eqnarray*}
Thus, as $S(0,p+n;cq)=S(0,(p+n)\overline{q};c)S(0,p+n;q)$ and $|S(0,(p+n)\overline{q};c)|\leq\sum_{l|(p+n,c)}l$,
\begin{eqnarray*}
K&\ll&p^{1/2}q^{-1+\varepsilon}\sum_{n=1}^{\infty}\frac{d(n)}{n}e^{-\sqrt{n}/2q^2}|S(0,p+n;q)|\sum_{c<q}\frac{\sum_{l|(p+n,c)}l}{c}\nonumber\\
&\ll&p^{1/2}q^{-1+\varepsilon}\sum_{n=1}^{\infty}\frac{d(n)d(p+n)}{n}e^{-\sqrt{n}/2q^2}|S(0,p+n;q)|.
\end{eqnarray*}
We break the sum over $n$ according to whether $q|(p+n)$ or $q\nmid(p+n)$. The contribution of the latter is $O(p^{1/2}q^{-1+\varepsilon})$. That of the former is
\begin{equation*}
\ll p^{1/2}q^{\varepsilon}\sum_{l=1}^{\infty}\frac{d(l)d(ql-p)}{ql-p}e^{-\sqrt{ql-p}/2q^2}\ll p^{1/2}q^{\varepsilon}(q-p)^{-1+\varepsilon}+p^{1/2}q^{-1+\varepsilon}.
\end{equation*}
The lemma follows.
\end{proof}

The case of $Y$ is more complicated as $Y_0$ is an oscillating function. For that we need the following standard lemma (for example, see [\textbf{\ref{KM}}]).

\begin{lemm}
Let $v\geq0$ and $J$ be a positive integer. If $f$ is a compactly supported $C^\infty$ function on $[Y,2Y]$, and there exists $\beta>0$ such that
\begin{equation*}
y^{j}f^{(j)}(y)\ll_{j}(1+\beta Y)^j
\end{equation*}
for $0\leq j\leq J$, then for any $\alpha>1$, we have
\begin{equation*}
\int_{0}^{\infty}Y_{v}(\alpha y)f(y)dy\ll\bigg(\frac{1+\beta Y}{1+\alpha Y}\bigg)^{J}Y.
\end{equation*}
\end{lemm}

\begin{lemm}
For $Y$ defined as in \eqref{4}, we have
\begin{equation*}
Y\ll p^{1/2}q^{-1+\varepsilon}.
\end{equation*}
\end{lemm}
\begin{proof} We have
\begin{equation}\label{6}
Y=\frac{4\pi^2}{q^2}\sum_{c<q}\frac{1}{c^2}\sum_{n=1}^{\infty}d(n)S(0,p-n;cq)y(n),
\end{equation}
where
\begin{equation}\label{7}
y(n)=\int_{0}^{\infty}Y_{0}\bigg(\frac{4\pi\sqrt{nt}}{cq}\bigg)J_{1}\bigg(\frac{4\pi\sqrt{tp}}{cq}\bigg)W_{q}\bigg(\frac{4\pi^2t}{q}\bigg)\xi(t)\frac{dt}{\sqrt{t}}.
\end{equation}
We make a smooth dyadic partition of unity that $\xi=\sum_{k}\xi_{k}$, where each $\xi_k$ is a compactly supported $C^\infty$ function on the dyadic interval $[X_k,2X_k]$. Moreover, $\xi_k$ satisfies $x^{j}\xi_{k}^{(j)}(x)\ll1$, for all $j\geq0$. We work on each $\xi_k$ individually, but we write $\xi$ instead of $\xi_k$ and, accordingly, $X$ rather than $X_k$.

By the change of variable $x:=2\sqrt{t}/cq$, we have
\begin{equation*}
y(n)=cq\int_{0}^{\infty}Y_{0}(2\pi\sqrt{n}x)J_{1}(2\pi\sqrt{p}x)W_{q}(\pi^2c^2qx^2)\xi\bigg(\frac{c^2q^2x^2}{4}\bigg)dx.
\end{equation*}
We define
\begin{equation*}
f(x):=J_{1}(2\pi\sqrt{p}x)W_{q}(\pi^2c^2qx^2)\xi\bigg(\frac{c^2q^2x^2}{4}\bigg).
\end{equation*}
This is a $C^\infty$ function compactly supported on $[\rho,2\rho]$, where $\rho=2\sqrt{X}/cq$.

We first treat the case $1/2\leq X\leq q$. We note that this involves $O(\log q)$ dyadic intervals. From Lemma 5 we have $x^{j}W^{(j)}(x)\ll_{j}\log q$ for $1/q\ll x\ll 1$. This, together with the recurrence relation $(x^{v}J_{v}(x))'=x^{v}J_{v-1}(x)$, gives
\begin{equation}\label{12}
x^{j}f^{(j)}(x)\ll_{j}(1+\sqrt{p}x)^{j}\log q.
\end{equation}
We are in a position to apply Lemma 7 to $f$ with $\alpha=2\pi\sqrt{n}$, $\beta=\sqrt{p}$ and $Y=\rho=2\sqrt{X}/cq$. The lemma yields, for any positive integer $J$,
\begin{equation}\label{8}
y(n)\ll cq\rho\bigg(\frac{1+\sqrt{p}\rho}{1+\sqrt{n}\rho}\bigg)^{J}\log q.
\end{equation}

Later, we will break the sum over $n$ in \eqref{6} in the following way
\begin{equation*}
\sum_{n\geq1}=\sum_{n\leq\rho^{-\kappa}}+\sum_{n>\rho^{-\kappa}},
\end{equation*}
where $\kappa>2$ will be chosen later. The estimate \eqref{8} will be used for $n>\rho^{-\kappa}$. We need another estimate for the range $n\leq\rho^{-\kappa}$. For this we go back to \eqref{7}, using $Y_{0}(x)\ll1+|\log x|$ and $J_{1}(x)\ll x$, to derive
\begin{equation}\label{9}
y(n)\ll\frac{\sqrt{p}X}{cq}(\log q)^2.
\end{equation}

We denote by $Y_1$ and $Y_2$ the corresponding splitted sums ($Y=Y_{1}+Y_{2}$). For the first sum, using \eqref{9}, we have
\begin{eqnarray}
Y_1&=&\frac{4\pi^2}{q^2}\sum_{c<q}\frac{1}{c^2}\sum_{n\leq\rho^{-\kappa}}d(n)S(0,p-n;cq)y(n)\nonumber\\
&\ll&p^{1/2}Xq^{-3+\varepsilon}\sum_{n\leq\rho^{-\kappa}}d(n)|S(0,p-n;q)|\sum_{c<q}\frac{1}{c^3}\sum_{l|(p-n,c)}l\nonumber\\
&\ll&p^{1/2}Xq^{-3+\varepsilon}\sum_{n\leq(\frac{q^2}{2\sqrt{X}})^{\kappa}}d(n)|S(0,p-n;q)|\sum_{\frac{2\sqrt{X}}{q}n^{1/\kappa}\leq c<q}\sum_{l|(p-n,c)}\frac{l}{c^3}\nonumber\\
&\ll&p^{1/2}q^{-1+\varepsilon}\sum_{n\leq(\frac{q^2}{2\sqrt{X}})^{\kappa}}\frac{d(n)d(p-n)}{n^{2/\kappa}}|S(0,p-n;q)|\nonumber\\
&\ll&p^{1/2}q^{2\kappa-5+\varepsilon}.\label{10}
\end{eqnarray}

For the second sum, we note that $\sqrt{p}\rho\ll1$ in this range. Using \eqref{8}, we have 
\begin{equation*}
y(n)\ll\sqrt{X}(\log q)n^{-J/2}\rho^{-J}.
\end{equation*}
Similarly to above, we deduce that
\begin{eqnarray}
Y_2&=&\frac{4\pi^2}{q^2}\sum_{c<q}\frac{1}{c^2}\sum_{n>\rho^{-\kappa}}d(n)S(0,p-n;cq)y(n)\nonumber\\
&\ll&\sqrt{X}q^{-2+\varepsilon}\sum_{n>\rho^{-\kappa}}\frac{d(n)}{n^{J/2}}|S(0,p-n;q)|\sum_{c<q}\frac{1}{c^2}\sum_{l|(p-n,c)}l\rho^{-J}\nonumber\\
&\ll&X^{-(J-1)/2}q^{J-2}\sum_{n}\frac{d(n)}{n^{J/2}}|S(0,p-n;q)|\sum_{l|p-n}l^{J-1}\sum_{c\leq\frac{2\sqrt{X}}{ql}n^{1/\kappa}}c^{J-2}\nonumber\\
&\ll&q^{-1+\varepsilon}\sum_{n}\frac{d(n)d(p-n)}{n^{J/2-(J-1)/\kappa}}|S(0,p-n;q)|.\label{11}
\end{eqnarray}

To this end, we choose $\kappa=2+\varepsilon/2$ and $J$ large enough so that $J/2-(J-1)/\kappa>1$. We hence obtain $Y_{1}\ll p^{1/2}q^{-1+\varepsilon}$ and, since the sum over $n$ in \eqref{11} converges, $Y_{2}\ll q^{-1+\varepsilon}$.

For $X>q$, similarly to \eqref{12}, using the bound $x^{j}W^{(j)}(x)\ll_{j}x^{-2}$, we have
\begin{equation*}
x^{j}f^{(j)}(x)\ll_{j}(1+\sqrt{p}x)^{j}q^{-2}(cx)^{-4}.
\end{equation*}
Lemma 7 then gives
\begin{equation*}
y(n)\ll cq\rho\bigg(\frac{1+\sqrt{p}\rho}{1+\sqrt{n}\rho}\bigg)^{J}\frac{q^2}{X^2}.
\end{equation*}
For the range $n\leq\rho^{-\kappa}$, a better bound than \eqref{9} in this case is
\begin{equation*}
y(n)\ll\frac{\sqrt{p}X}{cq}(\log q)\frac{q^2}{X^2}.
\end{equation*}
Since $q^2/X^2\ll1$, all the previous estimates remain valid. The ony place where this is not the case is the sum over $n\ll(q^2/2\sqrt{X})^{\kappa}$ in \eqref{10}. However, this sum is void for $X>q^4/4$ and the former estimate still works in the larger interval $X\leq q^4/4$. Also, the quantity saved $q^2/X^2$ is sufficient to allow the sum over the dyadic values of $X$ involved to converge. The lemma follows.
\end{proof}

The proof of the theorem is complete.

\end{document}